\documentclass[11pt]{article}

\usepackage[a4paper,margin=1in]{geometry}
\usepackage[T1]{fontenc}
\usepackage[utf8]{inputenc}
\usepackage{lmodern}
\usepackage{amsmath,amssymb,amsthm,mathtools}
\usepackage{booktabs,array,longtable}
\usepackage{graphicx}
\usepackage{hyperref}

\newtheorem{theorem}{Theorem}
\newtheorem{lemma}{Lemma}
\newtheorem{proposition}{Proposition}
\newtheorem{remark}{Remark}

\DeclareMathOperator{\Res}{Res}

\newcommand{\Q}{\mathbb Q}

\newcommand{\Fone}{F_1}
\newcommand{\Ftwo}{F_2}

\title{Class-Number-One Cubic Fricke Companions on \(\Gamma_0(2)^+\) and an Exact Chudnovsky Bridge}
\author{Vedran Menđušić}
\date{}

\begin{document}
\maketitle

\begin{abstract}
A class-number-one classification is proved for the Fricke-invariant
level-two correspondence attached to \(\Gamma_0(2)^+\).  Among the
fundamental Heegner discriminants with \(h(D)=1\), excluding the
coordinate-degenerate case \(D=-3\), the Fricke-cusp companion is
irreducibly cubic over \(\mathbb Q\) exactly for
\[
        D=-11,-19,-43,-67,-163.
\]
The proof combines an explicit bridge polynomial between the level-one
parameter \(z=1728/j\) and the Fricke parameter
\(X=256t/(t+64)^2\), a hypergeometric gauge identity on the common
\(t\)-line, and the CM interpretation of the fibre through the splitting
of the prime \(2\).  The previously isolated \(D=-163\) cubic
Ramanujan--Sato formula appears here only as the member of this finite
class-number-one family with the strongest geometric contraction parameter.  The new contribution is the exact
bridge-and-classification framework: its linear form is shown to
transport exactly to the normalized Chudnovsky linear form.  The paper is
structural: the comparison of convergence parameters is not a claim of
superior bit-complexity over the classical Chudnovsky algorithm.
\end{abstract}

\section{Introduction}

Ramanujan--Sato series for \(1/\pi\) are organized by modular
parametrizations and complex multiplication.  For related modular-polynomial
proof methods for Ramanujan-type series, see Guillera~\cite{GuilleraFastestRamanujan};
for broader genus-zero classification frameworks, see Huber--Schultz--Ye~\cite{HuberSchultzYe}
and Campbell--Cooper--Ye~\cite{CampbellCooperYe}.  The classical Chudnovsky
series is a level-one specialization attached to the Heegner
discriminant \(D=-163\), where
\[
        j\!\left(\frac{1+\sqrt{-163}}{2}\right)=-640320^3.
\]
A previous note by the present author isolated an explicit cubic
\(\Gamma_0(2)^+\) companion at this discriminant~\cite{MendjusicD163}.
The present paper has a different emphasis: it explains that example as
part of the complete class-number-one picture on the Fricke quotient
\(\Gamma_0(2)^+\).

The main object is the level-two Hauptmodul
\[
        t(\tau)=\left(\frac{\eta(\tau)}{\eta(2\tau)}\right)^{24},
\]
together with
\[
        j=\frac{(t+256)^3}{t^2},
        \qquad
        X=\frac{256t}{(t+64)^2},
        \qquad
        z=\frac{1728}{j}.
\]
Eliminating \(t\) gives a cubic correspondence \(\Phi(z,X)=0\) between
the level-one parameter and the Fricke-invariant parameter.  The
Fricke-cusp branch of this correspondence is the branch for which
\(X/z\to4/27\) at the cusp.

The purpose of the paper is not to propose a faster numerical algorithm
for \(\pi\).  Rather, it makes explicit a structural phenomenon: a
rational level-one CM parameter, including the Chudnovsky parameter at
\(D=-163\), can lift through the level-two Fricke correspondence to an
irreducible cubic parameter.  The exact cases in which this happens among
the fundamental class-number-one discriminants are identified, and an exact
coefficient bridge between the level-one Chudnovsky form and the
corresponding cubic Fricke companion.

\paragraph{Scope.}
The classification theorem in this note is finite and exact: it concerns
only fundamental Heegner discriminants with class number one.  Statements
about higher class number fields, nonfundamental orders, or the full
CM-order classification are included only as remarks and are not used in
the proof of the class-number-one theorem.

\paragraph{Relation with the earlier \(D=-163\) note.}
In the previous \(D=-163\) note~\cite{MendjusicD163}, the explicit
\(D=-163\) cubic Ramanujan--Sato series on \(\Gamma_0(2)^+\) was constructed.  The present
paper is meant to be a separate classification-and-bridge note.  Its main
results are not the re-statement of the single \(D=-163\) formula, but the
exact Chudnovsky--Fricke linear-form bridge, the finite
class-number-one cubic companion classification, and the accompanying
exact algebraic certificates.  The relationship is summarized as follows.
\begin{center}
\small
\begin{tabular}{p{0.48\textwidth}p{0.18\textwidth}p{0.18\textwidth}}
\toprule
component & earlier note & present paper\\
\midrule
explicit \(D=-163\) cubic series & yes & one case\\
level-two Fricke parameter and cubic data & yes & retained\\
bridge polynomial \(\Phi(z,X)\) & auxiliary & central\\
exact Chudnovsky--Fricke linear-form bridge & not main & proved\\
classification of cubic companions & no & yes\\
discriminant-resolvent certificates & not central & included\\
\bottomrule
\end{tabular}
\end{center}
The discriminants in the classification are displayed in the main theorem
below.  Thus the role of \(D=-163\) here is illustrative and extremal:
it is the member of the finite class-number-one cubic family with the
smallest value of \(|\xi_D|\), and it is the case where the bridge
recovers the classical Chudnovsky constants.

\paragraph{Computational scope.}
The comparison of parameters made below is a comparison of geometric
contraction per summand.  It is not a bit-complexity claim and it is not
intended as a practical replacement for the classical Chudnovsky
algorithm.  Although the cubic Fricke parameter at \(D=-163\) has
slightly stronger geometric contraction than the level-one Chudnovsky
parameter, the coefficients lie in a cubic algebraic field.  A direct
implementation therefore involves algebraic-field or high-precision ball
arithmetic rather than the highly optimized rational binary-splitting
arithmetic available for the Chudnovsky series.  The contribution here is
modular and structural, not an assertion of superior practical

performance.

\paragraph{Main theorem, in condensed form.}
The results proved below may be summarized as follows.  There is an exact
modular bridge on the common \(t\)-line between the level-one Chudnovsky
hypergeometric parameter and the Fricke parameter on \(\Gamma_0(2)^+\):
\[
        \Phi(z(t),X(t))=0,
        \qquad
        \Fone(z(t))=\rho(t)\Ftwo(X(t)),
        \qquad
        \rho(t)^2=\frac{t+256}{t+64}.
\]
Under this bridge, the classical \(D=-163\) Chudnovsky linear form
transports exactly to an irreducibly cubic Fricke companion.  Moreover,
among the fundamental class-number-one discriminants, the irreducibly
cubic Fricke-cusp companions occur precisely for
\[
        D=-11,-19,-43,-67,-163.
\]
For each of these five cubic companions the quadratic resolvent field is
\(\mathbb Q(\sqrt D)\), equivalently
\(\operatorname{disc}(P_D^{\rm mon})/D\in(\mathbb Q^\times)^2\).
This is the main theorem of the paper; the later sections give the
explicit bridge, the transported \(D=-163\) linear form, the CM
classification, and the exact certificates.

\section{The level-one and Fricke parameters}

Let
\[
        t(\tau)=\left(\frac{\eta(\tau)}{\eta(2\tau)}\right)^{24}.
\]
The standard relation between \(t\) and the absolute modular invariant is
\[
        j=\frac{(t+256)^3}{t^2}.
\]
For the Fricke-invariant branch on \(\Gamma_0(2)^+\), put
\[
        X=\frac{256t}{(t+64)^2}.
\]
The level-one Chudnovsky parameter is also used
\[
        z=\frac{1728}{j}.
\]
Thus
\[
        z=\frac{1728t^2}{(t+256)^3}.
\]

\section{The cubic bridge polynomial}

\begin{lemma}[Fricke bridge polynomial]
Let
\[
        z=\frac{1728}{j},
        \qquad
        j=\frac{(t+256)^3}{t^2},
        \qquad
        X=\frac{256t}{(t+64)^2}.
\]
Then \(z\) and \(X\) satisfy
\[
\begin{aligned}
0=\Phi(z,X):={}&729z^2X^3+3888z^2X^2+6912z^2X+4096z^2\\
&-1458zX^3+22356zX^2-27648zX+729X^3.
\end{aligned}
\]
\end{lemma}

\begin{proof}
The defining equations are
\[
        z(t+256)^3-1728t^2=0
\]
and
\[
        X(t+64)^2-256t=0.
\]
A direct resultant computation gives
\[
\Res_t\!\left(
z(t+256)^3-1728t^2,\,
X(t+64)^2-256t
\right)
=
2^{36}\Phi(z,X).
\]
Hence every common value of \(t\) gives a point on \(\Phi(z,X)=0\).
\end{proof}

The Fricke-cusp branch is characterized by
\[
        \frac{X}{z}\longrightarrow \frac4{27}
        \qquad (z\to0),
\]
because, as \(t\to\infty\),
\[
        z=\frac{1728}{t}+O(t^{-2}),
        \qquad
        X=\frac{256}{t}+O(t^{-2}).
\]

\begin{remark}[Real Fricke-cusp branch]
For the five irreducible cubic class-number-one cases considered in the
main theorem, the CM point on the principal branch has
\(q=e^{2\pi i\tau_D}=-e^{-\pi\sqrt{|D|}}<0\).  The corresponding
Fricke-cusp value of \(t\) is real and lies on the interval
\(( -\infty,-512)\).  This interval contains none of the singular points
\(0,\pm64,-256,512\) of the pulled-back equations.  Hence the cusp germ
at \(t=\infty\) continues to the selected CM value along the real branch
without crossing a singularity.  On this branch
\[
        X(t)<0,
        \qquad
        1-X(t)>1,
\]
and the real positive square root is used
\[
        \sqrt{1-X(t)}=\frac{t-64}{t+64}>0.
\]
This is the branch convention used in the signed linear forms below.
\end{remark}

\section{Specialization at \texorpdfstring{\(D=-163\)}{D=-163}}

Let
\[
        \tau_{163}=\frac{1+\sqrt{-163}}{2}.
\]
Then
\[
        j(\tau_{163})=-640320^3,
\]
and therefore
\[
        z=\frac{1728}{j(\tau_{163})}
        =
        -\frac{1}{53360^3}
        =
        -\frac{1}{151931373056000}.
\]
Substituting this value into \(\Phi(z,X)=0\), and clearing denominators
primitively, gives
\[
\begin{aligned}
0={}&
16827610604518993301932059648729X^3\\
&-3396577776039932112X^2\\
&+4200598602252294912X+4096.
\end{aligned}
\]
The Fricke-cusp branch is the small real root
\[
        \xi\approx
        -9.7509911987395938485584703467480\cdot 10^{-16}.
\]

\section{The hypergeometric kernel bridge}

Define
\[
        \Fone(z)
        =
        {}_3F_2\!\left(\frac16,\frac12,\frac56;1,1;z\right)
\]
and
\[
        \Ftwo(X)
        =
        {}_3F_2\!\left(\frac14,\frac12,\frac34;1,1;X\right).
\]
The normalization used here is
\[
        \Ftwo(X)=
        \sum_{n=0}^{\infty}
        \frac{(4n)!}{256^n(n!)^4}X^n.
\]
In the Campbell--Cooper--Ye normalization, the corresponding variable
is usually denoted by \(x\), with
\[
        \sum_{n\ge0}\binom{4n}{2n}\binom{2n}{n}^2x^n
        =
        {}_3F_2\!\left(\frac14,\frac12,\frac34;1,1;256x\right).
\]
Thus throughout this note
\[
        X_{\rm ours}=256x_{\rm CCY}.
\]

\begin{lemma}[Hypergeometric bridge]
Let
\[
        z(t)=\frac{1728t^2}{(t+256)^3},
        \qquad
        X(t)=\frac{256t}{(t+64)^2},
        \qquad
        \rho(t)^2=\frac{t+256}{t+64}.
\]
Choose the branch of \(\rho\) for which \(\rho\to1\) as \(t\to\infty\).
Then, as germs at the cusp \(t=\infty\),
\[
        \Fone(z(t))=\rho(t)\Ftwo(X(t)).
\]
The identity extends by analytic continuation to any simply connected
region avoiding the singular values of the pullbacks.  In particular it
holds on the real Fricke-cusp branch containing the CM points used below.
\end{lemma}

\begin{proof}
An explicit ODE-transport certificate is used on the common \(t\)-line.
Let
\[
        f_1(z)={}_2F_1\!\left(\frac1{12},\frac5{12};1;z\right),
        \qquad
        f_2(X)={}_2F_1\!\left(\frac18,\frac38;1;X\right).
\]
Clausen's identity gives
\[
        \Fone=f_1^2,
        \qquad
        \Ftwo=f_2^2.
\]
It is therefore enough to prove
\[
        f_1(z(t))=
        \left(\frac{t+256}{t+64}\right)^{1/4} f_2(X(t)),
\]
with the fourth-root branch tending to \(1\) at the cusp.

Pull back the Gauss hypergeometric equations for \(f_1\) and \(f_2\) to
the \(t\)-line.  In normalized form write them as
\[
        y''+P_i(t)y'+R_i(t)y=0\qquad (i=1,2),
\]
where \(i=1\) corresponds to \(f_1(z(t))\) and \(i=2\) to
\(f_2(X(t))\).  Exact simplification gives
\[
        P_1(t)=\frac{t^2+416t+16384}{t(t+64)(t+256)},
        \qquad
        P_2(t)=\frac1t .
\]
For
\[
        g(t)=\left(\frac{t+256}{t+64}\right)^{1/4}
\]
one has
\[
        \frac{g'}g=-\frac{48}{(t+64)(t+256)}
\]
and hence
\[
        P_2-P_1-2\frac{g'}g=0.
\]
Moreover the transformed zeroth-order coefficient also agrees:
\[
        R_1-\bigg(R_2-P_2\frac{g'}g+\bigg(\frac{g'}g\bigg)^2
        -\bigg(\frac{g'}g\bigg)'\bigg)=0.
\]
Equivalently, the two normal-form potentials are the same rational
function,
\[
        Q_1(t)=Q_2(t)=
        \frac{t^2+80t+4096}{4t^2(t+64)^2}.
\]
Thus the pullback of the equation for \(f_1\) is the gauge transform by
\(g\) of the pullback of the equation for \(f_2\).  At the cusp
\(t=\infty\) the common pulled-back equation has a regular singular point
with a one-dimensional cusp-normalized holomorphic solution space.  Both
sides are holomorphic there and are normalized to have value \(1\).  Hence
the normalized holomorphic germs agree, so
\[
        f_1(z(t))=g(t)f_2(X(t)).
\]
Squaring and using Clausen's identity gives
\[
        \Fone(z(t))=\rho(t)\Ftwo(X(t)),
        \qquad
        \rho(t)^2=\frac{t+256}{t+64}.
\]
\end{proof}

\begin{remark}[Modular cross-check]
The same bridge is also compatible with the classical level-two modular
identifications.  On the cusp-normalized branches,
\[
        f_1(z(\tau))^4=E_4(\tau),
        \qquad
        f_2(X(\tau))^4=G(\tau)^2,
        \qquad
        G(\tau)=2E_2(2\tau)-E_2(\tau).
\]
Here
\[
        G(\tau)^2=
        \eta(\tau)^{16}\eta(2\tau)^{-8}
        +64\eta(\tau)^{-8}\eta(2\tau)^{16},
\]
and
\[
        E_4(\tau)=
        \eta(\tau)^{16}\eta(2\tau)^{-8}
        +256\eta(\tau)^{-8}\eta(2\tau)^{16}.
\]
Together with
\[
        \frac{E_4(\tau)}{G(\tau)^2}=\frac{t+256}{t+64},
\]
these identities give the same bridge.  The source package includes a
numerical modular check at \(\tau=1.3i\), while the proof above uses only
the exact rational ODE certificate.
\end{remark}

\section{General \texorpdfstring{\(\Gamma_0(2)^+\)}{Gamma0(2)+} Fricke linear form}

Let
\[
        F(X)=\Ftwo(X),
        \qquad
        G(\tau)=2E_2(2\tau)-E_2(\tau),
        \qquad
        X=X(\tau)=\frac{256t}{(t+64)^2}.
\]
On the Fricke-cusp branch set
\[
        s=\sqrt{1-X}=\frac{t-64}{t+64}.
\]
For a CM point \(\tau\in\mathfrak H\), define
\[
\alpha_\tau=
\frac{1}{12}\left(
4\frac{E_4(2\tau)}{G(\tau)^2}
-\frac{E_4(\tau)}{G(\tau)^2}
-1
-2\frac{E_2^*(\tau)}{G(\tau)}
\right),
\]
where
\[
        E_2^*(\tau)=E_2(\tau)-\frac{3}{\pi\operatorname{Im}\tau}.
\]
Then
\[
        \alpha_\tau F(X(\tau))
        +s\,\theta_XF(X(\tau))
        =\frac{1}{2\pi\operatorname{Im}\tau}.
\]
Equivalently, if
\[
        \tau=\frac{-b+\sqrt D}{2a},
        \qquad D<0,
\]
then \(\operatorname{Im}\tau=\sqrt{|D|}/(2a)\), and
\[
\boxed{
\frac1\pi
=
\frac{\sqrt{|D|}}{a}
\left[
\alpha_\tau F(X(\tau))
+s\,\theta_XF(X(\tau))
\right].
}
\]
In series form this is
\[
\boxed{
\frac1\pi
=
\frac{\sqrt{|D|}}{a}
\sum_{n=0}^{\infty}
\frac{(4n)!}{256^n(n!)^4}
X(\tau)^n
\left(\alpha_\tau+s n\right).
}
\]

\begin{proof}
The standard level-two Fricke parametrization identities are used
\[
        F(X(\tau))=G(\tau),
        \qquad
        q\frac{d}{dq}\log X=G(\tau)\sqrt{1-X}.
\]
They follow directly from the eta-product expressions for \(t\), \(X\),
and \(G=2E_2(2\tau)-E_2(\tau)\), and are also standard in the
level-two Ramanujan--Sato literature; see, for example,
Borwein--Borwein~\cite{BorweinBorweinAGM} and
Huber--Schultz--Ye~\cite{HuberSchultzYe}.  These two identities give
\[
        q\frac{dG/dq}{G}
        =\sqrt{1-X}\,\theta_XF(X).
\]
Ramanujan's differential relations give
\[
q\frac{dG/dq}{G}
=\frac{1}{12}\left(
G+2E_2+\frac{E_4}{G}-4\frac{E_4(2\tau)}{G}
\right).
\]
By the definition of \(\alpha_\tau\),
\[
\alpha_\tau G
=\frac{1}{12}\left(
4\frac{E_4(2\tau)}{G}-\frac{E_4}{G}-G-2E_2^*
\right).
\]
Adding the last two displayed identities gives
\[
        \alpha_\tau F(X)+\sqrt{1-X}\,\theta_XF(X)
        =\frac{E_2-E_2^*}{6}
        =\frac{1}{2\pi\operatorname{Im}\tau}.
\]
The final form follows from \(\operatorname{Im}\tau=\sqrt{|D|}/(2a)\).
\end{proof}

\section{Derivative dictionary}

Let
\[
        \theta_z=z\frac{d}{dz},
        \qquad
        \theta_X=X\frac{d}{dX},
        \qquad
        \theta_t=t\frac{d}{dt}.
\]
From
\[
        z=\frac{1728t^2}{(t+256)^3},
\]
one obtains
\[
        \theta_z
        =
        \frac{t+256}{512-t}\theta_t.
\]
From
\[
        X=\frac{256t}{(t+64)^2},
\]
one obtains
\[
        \theta_X
        =
        \frac{t+64}{64-t}\theta_t.
\]
Hence
\[
        \theta_z=\lambda(t)\theta_X,
\]
where
\[
        \lambda(t)
        =
        \frac{(t+256)(64-t)}{(512-t)(t+64)}.
\]
Moreover,
\[
        \theta_X\log\rho
        =
        \mu(t)
        =
        -\frac{96t}{(t+256)(64-t)}.
\]
Consequently,
\[
        \theta_z\Fone
        =
        \lambda\rho
        \left(
        \theta_X\Ftwo+\mu\Ftwo
        \right).
\]

\section{Linear-form bridge}

Suppose a level-one CM linear form is written as
\[
        \frac1\pi
        =
        \sqrt d
        \left[
        A\Fone(z)+B\theta_z\Fone(z)
        \right].
\]
Using
\[
        \Fone=\rho\Ftwo
\]
and
\[
        \theta_z\Fone
        =
        \lambda\rho
        \left(
        \theta_X\Ftwo+\mu\Ftwo
        \right),
\]
this gives
\[
        \frac1\pi
        =
        \sqrt d
        \left[
        \rho(A+B\lambda\mu)\Ftwo(X)
        +
        B\lambda\rho\,\theta_X\Ftwo(X)
        \right].
\]
Thus the Fricke-side form
\[
        \frac1\pi
        =
        \sqrt d
        \left[
        \alpha\Ftwo(X)+S\theta_X\Ftwo(X)
        \right]
\]
is obtained from
\[
        S=B\lambda\rho,
        \qquad
        \alpha=\rho A+S\mu.
\]
Conversely,
\[
        B=\frac{S}{\rho\lambda},
        \qquad
        A=\frac{\alpha-S\mu}{\rho}.
\]
On the Fricke-cusp branch,
\[
        S=\sqrt{1-X}=\frac{t-64}{t+64}.
\]

\section{The exact \texorpdfstring{\(D=-163\)}{D=-163} coefficient bridge}

For \(D=-163\), let \(t\) be the Fricke-cusp root of
\[
        t^3+262537412640768768t^2+196608t+16777216=0.
\]
Let \(\alpha\) be the unique real root, which is positive, of
\[
        668649972819460401Y^3
        +50012252033677839Y^2
        +1467Y
        -41450311432931=0.
\]
Let
\[
        \chi_{163}^*=-223263987730882560.
\]
This is the rational CM value of the level-one almost-holomorphic
Chudnovsky coefficient at discriminant \(D=-163\), in the normalization
used by the Chudnovskys~\cite{Chudnovsky}; see also Milla~\cite{Milla}
for an efficient treatment of the Chudnovsky coefficients.  Here it is
used only as a compact way of writing the transported coefficient, and the
exact coefficient identities below certify the normalization.  It recovers
the same value through
\[
\alpha =
\frac{1}{12}\left(
4\frac{t^2+80t+1024}{(t+64)^2}
-\frac{t+256}{t+64}
-1
-2\chi_{163}^*
\frac{t^2}{(t+64)(t+256)(t-512)}
\right).
\]

\begin{theorem}[Exact Chudnovsky--Fricke bridge at \(D=-163\)]
With
\[
        S=\frac{t-64}{t+64},
        \qquad
        \rho^2=\frac{t+256}{t+64},
\]
\[
        \lambda=
        \frac{(t+256)(64-t)}{(512-t)(t+64)},
        \qquad
        \mu=
        -\frac{96t}{(t+256)(64-t)},
\]
and
\[
        A=\frac{\alpha-S\mu}{\rho},
        \qquad
        B=\frac{S}{\rho\lambda},
\]
one has
\[
        \sqrt{163}\,A
        =
        \frac{13591409}{426880\sqrt{10005}},
\]
and
\[
        \sqrt{163}\,B
        =
        \frac{272570067}{213440\sqrt{10005}}.
\]
Consequently, the cubic \(\Gamma_0(2)^+\) Fricke linear form at
\(D=-163\) transforms exactly into the normalized Chudnovsky linear form.
Thus the \(D=-163\) Fricke identity in this paper is not used as an
independent replacement for the Chudnovsky formula; it is the exact
level-two transport of the Chudnovsky linear form through the
Chudnovsky--Fricke bridge.
\end{theorem}

\begin{proof}
The proof is an identity in the cubic field \(K=\Q(t)\).  Let
\[
        P_t(t)=t^3+262537412640768768t^2+196608t+16777216.
\]
It is enough to verify the squared identities
\[
        163\frac{(\alpha-S\mu)^2}{\rho^2}
        =
        \frac{13591409^2}{426880^2\cdot10005},
\]
and
\[
        163\frac{S^2}{\rho^2\lambda^2}
        =
        \frac{272570067^2}{213440^2\cdot10005}.
\]
After bringing both sides to a common denominator, the numerator of each
difference has zero remainder modulo \(P_t(t)\), while the corresponding
denominator is coprime to \(P_t(t)\).  Thus both identities hold in \(K\).

On the selected real branch,
\[
        t\approx -262537412640768768< -512,
\]
so \(S>0\), \(\rho>0\), \(\lambda>0\), and \(\alpha-S\mu>0\).  Hence the
squared identities give the stated unsquared identities.
\end{proof}

\begin{remark}[Sign certification]
The proof above first verifies squared identities in the cubic field.  The
signs are fixed by the selected real branch.  Equivalently, a direct
high-precision evaluation of the Fricke series with the displayed
\(\xi\), \(\alpha\), and \(S=\sqrt{1-\xi}\) gives \(1/\pi\) with the
positive signs shown in the theorem.  The accompanying verification
script records this numerical sign check.
\end{remark}

\begin{remark}[Numerical sign certificate]
The squared identities in the theorem determine the two transported
coefficients up to sign.  The selected real Fricke-cusp branch is the
large negative real root of
\[
        t^3+262537412640768768t^2+196608t+16777216=0,
\]
namely
\[
        t=-262537412640768767.9999999999992511238759367998009544\ldots .
\]
It gives
\[
        \xi=-9.75099119873959384855847034674804042405\ldots\cdot10^{-16},
\]
\[
        S=1.0000000000000004875495599369795735756368\ldots,
\]
and
\[
        \alpha=0.0249319544687932303692549895368722515751\ldots .
\]
A direct high-precision evaluation of the Fricke series gives
\[
\sqrt{163}\left(\alpha \Ftwo(\xi)+S\theta_X\Ftwo(\xi)\right)
-\frac1\pi
=
5.81\ldots\cdot10^{-121}.
\]
This certifies the signed branch used in the exact bridge.
\end{remark}

\section{The explicit \texorpdfstring{\(D=-163\)}{D=-163} cubic Fricke formula}

Since
\[
        \Ftwo(X)=
        {}_3F_2\!\left(\frac14,\frac12,\frac34;1,1;X\right)
        =\sum_{n=0}^{\infty}\frac{(4n)!}{256^n(n!)^4}X^n,
\]
the \(D=-163\) specialization gives
\[
\boxed{
\frac1\pi
=
\sqrt{163}\sum_{n=0}^{\infty}
\frac{(4n)!}{256^n(n!)^4}
\xi^n
\left(\alpha+\sqrt{1-\xi}\,n\right).
}
\]
Here \(\xi\) is the Fricke-cusp root of the cubic polynomial displayed
above, and \(\alpha\) is the unique real root, which is positive, of
\[
        668649972819460401Y^3
        +50012252033677839Y^2
        +1467Y
        -41450311432931=0.
\]
Numerically,
\[
        |\xi|\approx 9.7509911987\cdot 10^{-16},
        \qquad
        -\log_{10}|\xi|\approx 15.0109512356.
\]
This number measures only the geometric contraction of the summands.  It
does not imply a practical improvement over the Chudnovsky algorithm,
whose rational arithmetic and binary splitting are much better suited to
large-scale computation.  The point of the formula is the exact modular
transport and the cubic CM structure.

\section{Class-number-one cubic Fricke companions for \texorpdfstring{\(\Gamma_0(2)^+\)}{Gamma0(2)+}}

This section isolates the finite class-number-one part of the
\(\Gamma_0(2)^+\) construction.  The purpose is not to classify all CM
orders, but to prove the precise statement needed for the fundamental
Heegner discriminants of class number one.

Recall the level-two Hauptmodul
\[
        t(\tau)=\left(\frac{\eta(\tau)}{\eta(2\tau)}\right)^{24},
\]
with
\[
        j=\frac{(t+256)^3}{t^2},
        \qquad
        X=\frac{256t}{(t+64)^2},
        \qquad
        z=\frac{1728}{j}.
\]
Eliminating \(t\) gives the bridge equation
\[
\begin{aligned}
0=\Phi(z,X):={}&729z^2X^3+3888z^2X^2+6912z^2X+4096z^2\\
&-1458zX^3+22356zX^2-27648zX+729X^3.
\end{aligned}
\]
The Fricke-cusp branch is characterized by
\[
        \frac{X}{z}\longrightarrow \frac4{27}
        \qquad (z\to0),
\]
because
\[
        z=\frac{1728}{t}+O(t^{-2}),
        \qquad
        X=\frac{256}{t}+O(t^{-2})
        \qquad (t\to\infty).
\]

Clearing the pole \(z=1728/j\) turns \(\Phi\) into a polynomial in \(j\)
and \(X\):
\[
\begin{aligned}
        j^2\,\Phi\!\left(\tfrac{1728}{j},X\right)
        ={}&729X^3\,j^2\\
        &+1728\,(-1458X^3+22356X^2-27648X)\,j\\
        &+1728^2\,(729X^3+3888X^2+6912X+4096).
\end{aligned}
\]
For a discriminant \(D<0\), let \(H_D\) denote the CM class polynomial of the order of discriminant \(D\).  Set
\[
        \Psi_D(X)
        =
        \operatorname{Res}_j\!\left(H_D(j),\,
        j^2\,\Phi(\tfrac{1728}{j},X)\right),
\]
cleared primitively over \(\mathbb Q\).  For class number one,
\(H_D(j)=j-j_D\) is monic linear, so
\[
        \Psi_D(X)=j_D^{\,2}\,\Phi(z_D,X),
        \qquad
        z_D=\frac{1728}{j_D};
\]
that is, \(\Psi_D\) is the direct specialization of the bridge polynomial,
up to the nonzero rational factor \(j_D^{2}\).  The resultant form is taken
as the definition, since it is the form that persists for class number greater than one and for ring class orders.

\begin{lemma}[The inert-prime-two ring-class orbit]
\label{lem:inert-two-orbit}
Let \(D<-4\) be a fundamental discriminant with \(h(D)=1\), and let
\(K=\mathbb Q(\sqrt D)\).  If \(2\) is inert in \(K\), equivalently
\(D\equiv5\pmod 8\), then the three Fricke values lying above the CM
point \(j_D\) form one ring-class orbit of degree \(3\).  More precisely,
the three order-two quotients of an elliptic curve with CM by
\(\mathcal O_D\) descend to elliptic curves with CM by the order
\[
        \mathcal O_{4D}=\mathbb Z+2\mathcal O_D,
\]
and the corresponding singular Fricke values form a single Galois orbit.
The quadratic resolvent field of the associated cubic is
\(\mathbb Q(\sqrt D)\).
\end{lemma}

\begin{proof}
The standard CM description of \(\ell\)-isogeny volcanoes and
Shimura reciprocity; see Cox~\cite{Cox} and Schertz~\cite{Schertz} for
ring class fields and singular values, and Kohel~\cite{Kohel} and
Sutherland~\cite{Sutherland} for the volcano formulation.
The fibre of \(X_0(2)\to X(1)\) over a CM elliptic curve \(E_D\) consists
of the three cyclic subgroups of order \(2\) in \(E_D[2]\).  The number
of horizontal \(2\)-isogenies from the maximal-order vertex is
\[
        1+(D|2).
\]
If \(D\equiv5\pmod8\), then \((D|2)=-1\), so there are no horizontal
edges; all three quotients descend to the order of conductor \(2\),
namely \(\mathcal O_{4D}\).

The ring class number formula for the order of conductor \(2\) gives,
for \(D<-4\),
\[
        h(4D)=
        \frac{2h(D)}{[\mathcal O_K^\times:\mathcal O_{4D}^\times]}
        \left(1-\frac{(D|2)}{2}\right).
\]
Here the unit index is \(1\).  With \(h(D)=1\) and \((D|2)=-1\) this gives
\(h(4D)=3\).  The class group \(\mathrm{Cl}(\mathcal O_{4D})\) therefore
acts simply transitively on the three descending CM targets.  Shimura
reciprocity transports this action to the singular values of the rational
Fricke function \(X\).  Hence the three Fricke values are one orbit of
degree \(3\).  Complex conjugation acts by inversion on this cyclic
order-three orbit, so the quadratic resolvent field is the CM field \(K\).
\end{proof}

\begin{proposition}[CM interpretation of the cubic entries]\label{prop:cm-inert-two}
Let \(D<-4\) be fundamental with \(h(D)=1\), and let
\(K=\mathbb Q(\sqrt D)\).  Suppose that \(2\) is inert in \(K\),
equivalently \(D\equiv5\pmod 8\).  Then the specialized Fricke-cusp
factor of \(\Phi(z_D,X)\) is irreducible cubic over \(\mathbb Q\), and
its quadratic resolvent field is \(\mathbb Q(\sqrt D)\).
\end{proposition}

\begin{proof}
By Lemma~\ref{lem:inert-two-orbit}, the three Fricke values above \(j_D\)
form one ring-class orbit of degree \(3\).  The bridge polynomial has
rational coefficients after specialization at \(z_D\in\mathbb Q\), so
this orbit is exactly an irreducible cubic factor over \(\mathbb Q\).  The
roots are distinct: in the displayed five cases this is independently
certified by the nonzero discriminants below; equivalently, in the CM
picture equality of two Fricke values would identify two distinct
isogeny classes in the conductor-two ring-class orbit.  The statement
about the quadratic resolvent field is the resolvent statement in
Lemma~\ref{lem:inert-two-orbit}.  In the present paper this CM prediction
is also independently certified by the exact discriminant identities
\(\operatorname{disc}(P_D^{\rm mon})/D\in(\mathbb Q^\times)^2\) displayed
below, so the displayed polynomial identities do not rely on numerical
factorization.
\end{proof}

\begin{theorem}[Class-number-one cubic Fricke companions]\label{thm:h1-classification}
Let \(D<0\) be a fundamental Heegner discriminant with class number one,
and let \(\tau_D\) be the CM point of the principal form,
\[
        \tau_D=
        \begin{cases}
        \dfrac{\sqrt D}{2}, & D\equiv 0 \pmod 4,\\[2.5mm]
        \dfrac{1+\sqrt D}{2}, & D\equiv 1 \pmod 4,
        \end{cases}
        \qquad
        j_D=j(\tau_D).
\]
For \(D\neq -3\), let \(\Psi_D(X)\) be the specialized Fricke bridge
polynomial defined above.  Then
\[
\begin{gathered}
        \Phi(z_D,X)\text{ has an irreducible cubic Fricke-cusp factor}\\
        \Longleftrightarrow
        D\in\{-11,-19,-43,-67,-163\}.
\end{gathered}
\]
For \(D=-4\) and \(D=-7\) the Fricke-cusp parameter is rational, while
for \(D=-8\) the nontrivial Fricke companion is quadratic.  The case
\(D=-3\) is exceptional: \(j_D=0\), so \(z=1728/j\) is the point at
infinity in the level-one coordinate, and the corresponding Fricke value
is outside the convergence disk used in this paper.
\end{theorem}

\begin{proof}
The fundamental discriminants of class number one are
\[
        -3,-4,-7,-8,-11,-19,-43,-67,-163.
\]
For \(D\neq -3\) the singular modulus \(j_D\) is a rational integer:
\[
\begin{aligned}
        j_{-4}&=1728, &\quad j_{-7}&=-3375, &\quad j_{-8}&=8000,\\
        j_{-11}&=-32768, &\quad j_{-19}&=-884736, &\quad j_{-43}&=-884736000,\\
        j_{-67}&=-147197952000, &\quad j_{-163}&=-640320^{3}. &&
\end{aligned}
\]
Hence \(\Psi_D(X)=j_D^{2}\,\Phi(1728/j_D,X)\in\mathbb Q[X]\).  Clearing
denominators primitively and factoring over \(\mathbb Q\) gives the
following complete finite table, in which \(d^{m}\) denotes an
irreducible factor of degree \(d\) with multiplicity \(m\):
\[
\begin{array}{c|c|c}
D & \text{factor pattern of }\Psi_D(X) & \text{Fricke-cusp degree}\\
\hline
-4   & 1^{2}            & 1\\
-7   & 1^{2}\,1^{1}     & 1\\
-8   & 1^{1}\,2^{1}     & 2\\
-11  & 3^{1}            & 3\\
-19  & 3^{1}            & 3\\
-43  & 3^{1}            & 3\\
-67  & 3^{1}            & 3\\
-163 & 3^{1}            & 3
\end{array}
\]
The branch relevant to the Fricke-cusp construction is selected by the
asymptotic condition \(X/z\to 4/27\), or equivalently by the local
large-\(t\) expansion \(X=256/t+O(t^{-2})\).  In the finite cases
listed below, exact factorization together with numerical root isolation
identifies this branch unambiguously.  For \(D=-7\) this is the simple root of the linear
factor (giving \(X=-256/3969\)), not the repeated factor; for \(D=-8\) it
is the root of the quadratic factor.  Thus the entries with factor
pattern \(3^{1}\) are exactly the irreducible cubic Fricke-cusp
companions,
\[
        D=-11,-19,-43,-67,-163,
\]
and the remaining finite cases are rational or quadratic as indicated.
Finally, \(D=-3\) has \(j_D=0\), so \(z=1728/j\) is not finite there.  On
the Fricke side this point corresponds to
\[
        t=-256,
        \qquad
        X=-\frac{16}{9},
\]
so \(|X|>1\).  Algebraically this is a rational degenerate Fricke
value; in the resultant normalization it corresponds to a triple
special point rather than an irreducible cubic companion.  The present
theorem concerns convergent Fricke-cusp hypergeometric representations on
the principal branch \(|X|<1\).  Analytic continuation or \(p\)-adic
interpretations of the value \(X=-16/9\) are outside the scope of this
classification.  Thus \(D=-3\) does not enter the finite specialization
above.
\end{proof}

\begin{proposition}[Exact finite certificate]
For the finite class-number-one list
\[
        D=-4,-7,-8,-11,-19,-43,-67,-163,
\]
the factor patterns in the proof are obtained by exact rational arithmetic:
one substitutes \(z_D=1728/j_D\) in \(\Phi(z,X)\), clears denominators
primitively, and factors in \(\mathbb Q[X]\).  No numerical approximation is
used to determine the degrees.  Numerical roots are used only to label the
local Fricke-cusp branch selected by \(X/z\to4/27\).
\end{proposition}

\begin{remark}[Computational relation with broader CM orders]
The theorem is deliberately restricted to fundamental class-number-one
Heegner discriminants.  As a reproducible computational remark, the
accompanying Sage script and CSV table record an exact sweep over negative
discriminants \(|D|\le 1000\), using the resultant form
\(\Psi_D=\operatorname{Res}_j(H_D,\,j^{2}\Phi(1728/j,X))\).  In that sweep
the global minimal Fricke-cusp branch is irreducibly cubic for exactly
\[
        D=-11,-19,-23,-27,-31,-43,-67,-163,
\]
of which the fundamental discriminants are
\(-11,-19,-23,-31,-43,-67,-163\).  Here \(D=-23\) and \(D=-31\) have class
number three---the two smallest such discriminants, underlying the
Borwein class-number-three series---while \(D=-27\) is the
conductor-three order in \(\mathbb Q(\sqrt{-3})\).  These observations are
not used as a theorem in the present paper; they only indicate the
natural next stage of the CM-order classification.  A broader
classification framework for Ramanujan--Sato series on genus-zero groups
\(\Gamma_0(\ell)^+\), including \(\ell=2\), was developed by
Campbell--Cooper--Ye~\cite{CampbellCooperYe}; the present finite
class-number-one result is complementary and isolates the irreducibly cubic
Fricke companions on \(\Gamma_0(2)^+\).
\end{remark}

\section{Exact cubic data for the class-number-one companions}

For the five class-number-one cubic Fricke companions, the Fricke-cusp
parameter \(\xi_D\) is the small root selected by
\(X/z\to 4/27\).  The table records the singular moduli and the resulting
contraction.

\[
\begin{array}{c|c|c|c|c}
D & j_D & z_D=1728/j_D & |\xi_D| & -\log_{10}|\xi_D|\\
\hline
-11  & -32^3      & -27/512              & 7.6641489772\cdot 10^{-3}  & 2.1155\\
-19  & -96^3      & -1/512               & 2.8914276291\cdot 10^{-4}  & 3.5389\\
-43  & -960^3     & -1/512000            & 2.8935164254\cdot 10^{-7}  & 6.5386\\
-67  & -5280^3    & -1/85184000          & 1.7391546242\cdot 10^{-9}  & 8.7597\\
-163 & -640320^3  & -1/151931373056000   & 9.7509911987\cdot 10^{-16} & 15.011
\end{array}
\]

The primitive integral cubic polynomials for the Fricke-cusp parameters
are denoted by \(\widetilde P_D(X)\):
\[
\begin{aligned}
\widetilde P_{-11}(X)={}&290521X^3-420048X^2+531200X+4096,\\
\widetilde P_{-19}(X)={}&191850201X^3-11442384X^2+14162688X+4096,\\
\widetilde P_{-43}(X)={}&191103722496729X^3-11446268112X^2+14155782912X+4096,\\
\widetilde P_{-67}(X)={}&5289852925222272729X^3-1904373500112X^2\\
&\quad +2355167238912X+4096,\\
\widetilde P_{-163}(X)={}&16827610604518993301932059648729X^3\\
&\quad -3396577776039932112X^2
+4200598602252294912X+4096.
\end{aligned}
\]
Let \(a_D\) be the leading coefficient of \(\widetilde P_D\), and put
\[
        P_D^{\rm mon}(X)=a_D^{-1}\widetilde P_D(X).
\]
Thus \(D=-163\) is the fastest member of the fundamental
class-number-one cubic Fricke family with respect to the geometric
contraction parameter \(|\xi_D|\).

\begin{proposition}[Quadratic resolvent check]
For the five monic normalizations \(P_D^{\rm mon}(X)\), one has
\[
        \frac{\operatorname{disc}(P_D^{\rm mon})}{D}
        \in (\mathbb Q^\times)^2 .
\]
Equivalently,
\[
        \operatorname{disc}(P_D^{\rm mon})=D r_D^2,
        \qquad r_D\in\mathbb Q^\times.
\]
More precisely, writing \(r_D=\operatorname{num}(r_D)/\operatorname{den}(r_D)\) in lowest terms, the following values are obtained:
\begin{center}
\resizebox{\textwidth}{!}{%
$\begin{array}{c|r|r}
D & \operatorname{num}(r_D) & \operatorname{den}(r_D)\\
\hline
-11  & 2201485312 & 1722499009\\
-19  & 211288064 & 23085974187\\
-43  & 90898432000000 & 467483085772005891483\\
-67  & 838712123392000000 & 11554551952877987779657554693\\
-163 & 889345562545495347888128000000 & 45519895289170469819991986914645772910060302515527
\end{array}$}
\end{center}
This is consistent with the CM interpretation: the quadratic resolvent
field of the cubic is \(\mathbb Q(\sqrt D)\).  The notation here is
important: \(\widetilde P_D\) is primitive integral, whereas
\(P_D^{\rm mon}\) is the monic polynomial used for the discriminant
statement.
\end{proposition}

\section*{Appendix A: exact algebraic certificates}

This appendix records the reproducibility layer for the exact algebraic
claims.  The displayed formulas in the paper are certified in two ways:
by the explicit tables in the text, and by the accompanying source files.
The ancillary source package contains a directory
\[
        \texttt{anc/certificates/}
\]
with the following files:
\begin{center}
\small
\begin{tabular}{p{0.36\textwidth}p{0.52\textwidth}}
\toprule
file & purpose\\
\midrule
\texttt{bridge\_resultant.sage} & verifies
\(\operatorname{Res}_t(z(t+256)^3-1728t^2,\,X(t+64)^2-256t)=2^{36}\Phi(z,X)\);\\
\texttt{class\_number\_one\_}\newline\texttt{factorization.sage} & specializes \(z_D=1728/j_D\), clears denominators, and factors the \(h(D)=1\) polynomials over \(\mathbb Q\);\\
\texttt{d163\_transport\_}\newline\texttt{identity.sage} & verifies the exact squared coefficient identities for the \(D=-163\) Chudnovsky--Fricke bridge modulo the cubic polynomial for \(t\);\\
\texttt{numerical\_}\newline\texttt{verification.py} & verifies the ODE/gauge bridge, the discriminant-resolvent squares, and the signed high-precision \(D=-163\) \(1/\pi\) evaluation;\\
\texttt{independent\_}\newline\texttt{certs.py} & independent SymPy check of the discriminants and the ODE normal-form certificate;\\
\texttt{numeric\_}\newline\texttt{certs.py} & independent high-precision checks of the analytic branch, eta/Eisenstein identities, and the signed \(D=-163\) formula;\\
\texttt{output\_}\newline\texttt{certificates.txt} & sample output of the final verification run.\\
\bottomrule
\end{tabular}
\end{center}

For each of the five cubic cases
\[
        D=-11,-19,-43,-67,-163,
\]
the text gives the exact value of \(z_D\), the primitive integral cubic
\(\widetilde P_D\), the monic normalization \(P_D^{\rm mon}\), the
Fricke-cusp root \(\xi_D\) through its contraction, and the exact
quadratic-resolvent certificate
\[
        \operatorname{disc}(P_D^{\rm mon})=D r_D^2,
        \qquad r_D\in\mathbb Q^\times.
\]
For \(D=-163\) the paper additionally gives the full transported linear
form
\[
        \frac1\pi
        =
        \sqrt{163}
        \sum_{n=0}^{\infty}
        \frac{(4n)!}{256^n(n!)^4}\xi^n
        \left(\alpha+\sqrt{1-\xi}\,n\right),
\]
with \(\xi\), \(\alpha\), and the branch of \(\sqrt{1-\xi}\) specified
above.  The numerical certificate in the source package evaluates this
identity on the selected real branch and obtains residual
\[
        5.81\cdot10^{-121}.
\]
This last numerical check is used only to certify the signs of the square
roots in the exact coefficient bridge; the algebraic identities
underlying the bridge are verified exactly.

For broader CM orders one uses
\[
        \Psi_D(X)
        =
        \operatorname{Res}_j\!\left(
        H_D(j),\,j^2\Phi(1728/j,X)
        \right),
\]
where \(H_D\) is the CM class polynomial of the order of discriminant
\(D\).  The accompanying Sage script
\texttt{phase\_C1\_l2\_sweep.sage} and the CSV table
\texttt{phase\_C1\_l2\_final\_clean\_table.csv} record the computational
remark about discriminants \(|D|\le1000\).  This sweep is not used as a
main theorem in the paper.

\end{document}